\documentclass[12pt]{amsart}

\usepackage{fancyhdr,amsmath,amssymb,latexsym,verbatim,tikz,mathrsfs}
\usepackage[total={6in,9in},centering,includefoot,includehead]{geometry}
\usepackage{graphicx}
\usepackage{caption, comment}
\usepackage{subcaption}
\usetikzlibrary{arrows}
\usetikzlibrary{arrows, automata, positioning, quotes}

\begingroup
 
\newtheorem*{thm*}{Theorem}
\newtheorem{thm}[subsection]{Theorem} 

\newtheorem{lemma}[subsection]{Lemma}       
\newtheorem{prop}[subsection]{Proposition} 
\newtheorem{defn}[subsection]{Definition}
\newtheorem{ex}[subsection]{Example}

\newtheorem{remark}{Remark}  

\endgroup

\newcommand{\Z}{{\mathbb{Z}}}

\newcommand{\N}{\mathcal{N}}
\newcommand{\C}{\mathscr{C}}

\newcommand{\Aut}{\operatorname{Aut}}
\newcommand{\sC}{\mathscr{C}}
\newcommand{\CPA}{\operatorname{CPA}}
\newcommand{\GLA}{\operatorname{GLA}}
\newcommand{\Aff}{\operatorname{Aff}}
\newcommand{\id}{\operatorname{id}}
\newcommand{\Stab}{\operatorname{Stab}}
\newcommand{\xRightarrow}[2][]{\ext@arrow 0359\Rightarrowfill@{#1}{#2}}
\makeatother

\title[]{Automorphism Groups of nilpotent Lie algebras associated to  certain  graphs}

\author{Debraj Chakrabarti}
\address{Department of Mathematics, Central Michigan University, Mount Pleasant, MI 48859}
\email{chakr2d@cmich.edu}

\author{Meera Mainkar}
\address{Department of Mathematics, Central Michigan University, Mount Pleasant, MI 48859}
\email{maink1m@cmich.edu}

\author{Savannah Swiatlowski}
\address{Department of Mathematics, Central Michigan University, Mount Pleasant, MI 48859}
\email{swiat1sv@cmich.edu}

\thanks{Debraj Chakrabarti and Savannah Swiatlowski were partially supported by NSF grant DMS-1600371.}

\begin{document}

\begin{abstract}
We consider a family of 2-step nilpotent Lie algebras associated to uniform complete graphs on odd number of vertices. We prove that the symmetry group of such a graph is the holomorph of the additive cyclic group $\Z_n$. Moreover, we prove that the (Lie) automorphism group of the corresponding nilpotent Lie algebra contains the dihedral group of order $2n$ as a subgroup.


\end{abstract}

 \maketitle

\section{Introduction} 
Many  classes of $2$-step nilpotent Lie algebras associated with various types of graphs have been studied recently from different points of view, see,  e.g.,  \cite{DM, DDM, GGI, F, FJ, M, N, PS, PT, R, LW}.  A {\em 2-step nilpotent Lie algebra} is a Lie algebra where each 3-fold Lie bracket  $[X, [Y, Z]]$ of elements $X, Y, Z$ of the Lie algebra is 0. The 3-dimensional Heisenberg Lie algebra is well-studied example of a 2-step nilpotent Lie algebra.  In \cite{DM}, the authors studied the automorphism group of a $2$-step nilpotent Lie algebra associated with  a simple graph  and  then classified the graphs which correspond to the 2-step nilpotent Anosov Lie algebras.  These  Lie algebras give rise to interesting hyperbolic dynamics on nilmanifolds.

In this paper, we consider a similar problem for an interesting class of edge-colored directed simple graphs $H_n$ where $n$ is an odd integer.  We begin by considering the  underlying {\em undirected} edge-colored graph $G_n$ of $H_n$.  The example $G_5$ occurred in the recent paper \cite[Example 5.7]{PS}, in connection with uniform Lie algebras. The graphs $G_n$ are  remarkable for having a large amount of symmetry which can be used for constructing other objects associated with it with nontrivial symmetry, e.g. Einstein solvmanifolds \cite{D}, infranilmanifolds  etc.  Interestingly, we found that,  $G_n$ arises naturally when we consider the cyclic group $\Z_n$ of $n$ elements  as a space on which $\Z_n$ acts by translations, i.e. as a {\em torsor} without a distinguished identity element.  In section \ref{G_n} we will give the  algebraic definition of $G_n$  but now we introduce it geometrically. For every odd integer $n$ we construct the edge-colored simple graph $G_n$ by  beginning with the complete graph on $n$ vertices $v_1, \ldots, v_n$ and thinking of these vertices as the vertices of a regular $n$-gon in the plane. We  color the edges with $n$ colors $c_1, \ldots c_n$ in such a way that  for every vertex $v_k$ the $\dfrac{n-1}{2}$ edges which are perpendicular to the axis of symmetry of the $n$-gon passing through $v_k$ are colored with the same color $c_k$. Below we illustrate this for $n = 5$.

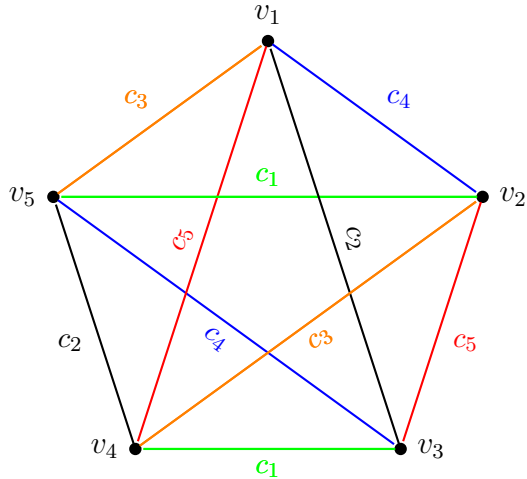
\begin{figure}[h!]
\begin{tikzpicture}[-,>=stealth',shorten >=1pt,auto,
  thick,vertex/.style={circle,draw,fill,scale=.35,font=\sffamily\large\bfseries},node distance=3in,thick]

\node[vertex, label=above:{$v_1$}](X0)  at (90:3) {};
  \node[vertex, label=right:{$v_2$}](X1) at (18:3) {};
 \node[vertex, label=right:{$v_3$}] (X2)  at (306:3)  {};
 \node[vertex, label=left:{$v_4$}] (X3) at (234:3) {};
	 \node[vertex, label=left:{$v_5$}] (X4) at (162:3) {};

\path[every node/.style={font=\sffamily\small}]
    (X0) edge ["$c_4$",blue] (X1)
					 (X1) edge ["$c_5$",red] (X2)
		 (X2) edge ["$c_1$",green] (X3)
         (X2) edge ["$c_1$",green] (X3)
         (X3) edge ["$c_2$"] (X4)
         (X4) edge ["$c_3$", orange] (X0)
         (X4) edge ["$c_3$", orange] (X0)
        (X4) edge [blue]  node[sloped, below] {$c_4$}(X2)
  (X0) edge [red] node[sloped, above] {$c_5$} (X3)
	(X1) edge [green] node[sloped, above] {$c_1$}  (X4)	
    (X1) edge [green] node[sloped, above] {$c_1$} (X4)	
    (X2) edge []  node[sloped, above] {$c_2$}(X0)
     (X3) edge [orange]  node[sloped, below] {$c_3$}(X1)
     (X3) edge [orange]  node[sloped, below] {$c_3$}(X1)
    ;
\end{tikzpicture}
\end{figure}


In our first result,   we compute explicitly the group of symmetries $\CPA(G_n)$ of $G_n$  preserving the coloring structure which we call the {\em color permuting automorphisms}, see Definition \ref{def-cpa}. 

\begin{thm}\label{main1}
{\rm  $\CPA(G_n) \cong \Z_n \rtimes \Aut(\Z_n)$.}
 \end{thm}

Here $\Aut(\Z_n)$ is the group of group-automorphisms of the cyclic group $\Z_n$. 
The semidirect product $\Z_n \rtimes \Aut(\Z_n)$ is known as the {\em holomorph} of $\Z_n$. 

This result is interesting because the graph $G_n$ is constructed out of the cyclic group $\Z_n$ and therefore the result expresses an aspect of the combinatorics  of this familiar object. 

As already mentioned, the real motivation for  considering these graphs comes from the theory of 2-step nilpotent Lie algebras. This idea goes back to \cite{DM} for simple graphs and has been extended by many authors  \cite{DM, DDM, GGI, F, FJ, M, N, PS, PT, R, LW}. In \cite{R, PS}, with each directed edge-colored graph $G$, a $2$-step nilpotent Lie algebra $\mathcal{N}_G$ was  associated and its properties were studied.   This construction is recalled in Section \ref{sec-graph} below.   
These Lie algebras can be thought as a quotient of the 2-step nilpotent Lie algebras associated with simple graphs as in \cite{DM}. In order to obtain a Lie algebra from an edge-colored graph, we will further need that the edges are directed. 
We assign a certain natural orientation of the edges to $G_n$ to obtain the  directed edge-colored simple graphs $H_n$ in Section \ref{sectionedgedirection}. We are interested in understanding the group $\Aut(\mathcal{N}_{H_n})$ of Lie algebra automorphisms of the corresponding Lie algebra $\mathcal{N}_{H_n}$. In the situation considered in \cite{DM}, each graph automorphism gives rise to a Lie algebra automorphism of the corresponding 2-step nilpotent Lie algebra. However, if we allow the repetition  of the edge-colors, then only certain type of graph automorphisms or symmetries can be extended to the automorphisms of the associated Lie algebra. We call those automorphisms as {\em graph Lie automorphisms} and we denote the group of all such automorphisms of a graph $G$ by $\GLA(G)$.  We compute explicitly the group $\GLA(G_n)$ and prove the following theorem. 

\begin{thm}\label{main2}
{\rm $\GLA(H_n) \cong D_n$, dihedral group of order $2n$.  Consequently $\Aut(\mathcal{N}_{H_n})$ contains a subgroup isomorphic to the dihedral group of order $2n$. }
\end{thm}

If $H_n$ is thought as above to be a regular $n$-gon in the plane along with all the diagonals which are colored and directed in a certain way, then $\GLA(H_n)$ can be thought of the Euclidean group of symmetries of this polygon, which is well-known to be the dihedral group of order $2n$.

The automorphism group of a  nilpotent Lie algebra plays an important role in studying  certain  Einstein homogeneous spaces  and  infranilmanifolds (see, e.g. \cite{DD,DV,LW}).  It would be very interesting to find a complete description of the automorphism group of $\N_{H_n}$.

\section{Edge-colored graphs  and  their Automorphisms}\label{Edge-ColoredGraphs}
In this section we recall some definitions (see \cite{PS} for example). Let $(S, E)$ denote a finite simple graph where $S$ is the set of vertices and $E$ is the set of edges.  We denote an edge by a 2-set $\{ \alpha, \beta \}$. Let $\C$ denote a finite set of {\em colors}. An {\em edge-coloring } 
is a surjective function $c : E \to \C$. We call a graph $G = (S, E, c: E \to \C)$  an {\em edge-colored} graph. 

Recall that a bijection $\sigma : S \to S$ is a {\em graph automorphism of } $(S, E)$ if the following holds: For all $\alpha, \beta \in S$, $\{\sigma(\alpha),  \sigma(\beta)\} \in E$ if and only if $\{\alpha, \beta\} \in E$. In this case, we extend $\sigma$ on the set $E$ by defining $\sigma(\{\alpha, \beta\}) = \{\sigma(\alpha), \sigma(\beta)\}$.

\begin{defn}\label{def-cpa}
{\rm Let $G = (S, E, c: E \to \C)$ be an edge-colored graph. A graph automorphism $\chi$ of $(S, E)$ is called a {\em color permuting automorphism} of $G$ if there exists a permutation $\phi$ of the set of colors  $\C$ such that $\phi \circ c = c \circ \chi$ on $E$.  The set of all color permuting automorphisms form a group which we denote by $\CPA(G)$. }
\end{defn}

\begin{ex}\label{edgecolored}
{\rm  Let $C_4$ denote a cycle graph on 4 vertices where the vertex set $S=  \{\alpha, \beta, \gamma, \delta\}$ and $E = \{\{\alpha, \beta\}, \{\beta, \gamma\}, \{\gamma, \delta\}, \{\alpha, \delta\}\}$. Let $\C = \{1, 2\}$. We define the edge-coloring $c : E \to \C$ by 
\[ c\left( \{\alpha, \beta\} \right) =   c\left(  \{\beta, \gamma\} \right) =  1\]
\[c\left( \{\gamma, \delta\} \right) =   c\left(  \{\delta, \alpha\} \right) =  2.\]

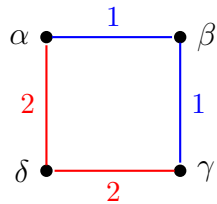
\begin{figure}[h!]
\begin{tikzpicture}[-,>=stealth',shorten >=1pt,auto,
  thick,vertex/.style={circle,draw,fill,scale=.35,font=\sffamily\large\bfseries},node distance=2in,thick]

\node[vertex, label=left:{$\alpha$}](X1) {};
  \node[vertex, label=right:{$\beta$}](X2) [right of=X1] {};
  \node[vertex, label=right:{$\gamma$}] (X3) [below of=X2]  {};
  \node[vertex, label=left:{$\delta$}] (X5)[below  of=X1]  {};

  \path[every node/.style={font=\sffamily\small}]
    (X1) edge ["$1$",blue] (X2)
		(X2)     edge ["$1$",blue]  (X3)
			(X3)		edge ["$2$",red]  (X5)
	(X5) edge["$2$", red] (X1)
				;
\end{tikzpicture}
\end{figure}
Let $\chi$ be  the permutation of $S$ given by $\chi = (\alpha \,\, \gamma)(\beta \,\, \delta)$, which is a graph automorphism of $(S,E)$.  Then the permutation of colors $\phi = (1\,\,2)$ satisfies $\phi \circ c = c \circ \chi$ on $E$ and hence $\chi$ is a color permuting automorphism of $C_4$ with the above coloring $c$. 

Note that  $\tau = (\alpha \,\, \beta \,\, \gamma \,\, \delta)$ is not a color permuting automorphism  of $C_4$ because \[c\left( \{\tau(\alpha), \tau(\beta)\} \right) = 1 \neq c\left( \{\tau(\beta), \tau(\gamma)\} \right).\]

It can be seen that $\CPA(C_4) = \{\id, (\alpha \,\, \gamma), (\beta \,\, \delta), (\alpha \,\, \gamma)(\beta \,\, \delta) \}$. \hfill$\square$

}
\end{ex}

A uniform graph is a special type of an edge-colored graph.  

\begin{defn}
{\rm We say that an edge-colored graph 
$(S, E, c: E \to \C)$ is a {\em uniform graph} if it satisfies the following properties.
\begin{enumerate}
\item No two  edges incident on the same vertex  have the same color, i.e. $c(\{\alpha, \beta \})\neq c(\{\alpha, \gamma\})$ if $\beta \neq \gamma$.
\item Each color occurs the same number of times, i.e. $|c^{-1}(\{c_i\}| = |c^{-1}(\{c_j\}|$ for all $c_i, c_j \in \C$. 
\end{enumerate}
}
\end{defn} 

\begin{ex}\label{4uniform}
{\rm  Consider the same uncolored graph $(S, E)$ as in Example \ref{edgecolored} and the same set of colors      $\C = \{1, 2\}.$ We define a new  edge-coloring $c : E \to \C$ by 
\[ c\left( \{\alpha, \beta\} \right) =   c\left(  \{\gamma, \delta\} \right) =  1\]
\[c\left( \{\beta, \gamma\} \right) =   c\left(  \{\delta, \alpha\} \right) =  2.\] 
We can see that  $G= (S, E, c: E \to \C)$ is a uniform graph and  $\CPA(G) \cong D_8$, the dihedral group with 8 elements.

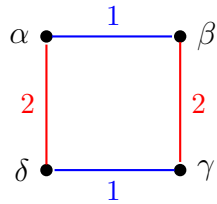
\begin{figure}[h!]
\begin{tikzpicture}[-,>=stealth',shorten >=1pt,auto,
  thick,vertex/.style={circle,draw,fill,scale=.35,font=\sffamily\large\bfseries},node distance=2in,thick]

\node[vertex, label=left:{$\alpha$}](X1) {};
  \node[vertex, label=right:{$\beta$}](X2) [right of=X1] {};
  \node[vertex, label=right:{$\gamma$}] (X3) [below of=X2]  {};
  \node[vertex, label=left:{$\delta$}] (X5)[below  of=X1]  {};

  \path[every node/.style={font=\sffamily\small}]
    (X1) edge ["$1$",blue] (X2)
		(X2)     edge ["$2$", red]  (X3)
			(X3)		edge ["$1$", blue]  (X5)
	(X5) edge["$2$", red] (X1)
				;
\end{tikzpicture}
\end{figure}

The edge-colored graph in Example \ref{edgecolored} is not uniform. 
}

\end{ex}


\section{A  uniform graph associated to $\Z_n$}\label{G_n}

In this section we associate an edge-colored graph to  a cyclic group of odd order and compute its symmetries. We note however, that this construction and the computation are very general and can be done for any abelian group of odd order. 

Throughout we assume that  $n$ is an odd integer.  We let $\Z_n$ denote the cyclic group of order $n$ written additively and denote the elements of $\Z_n$ as $\{0, \ldots, n-1\}$.  Let $G_n$ be an edge-colored graph where the underlying uncolored graph is  a complete graph  with  vertex set  $\Z_n$ and where the set of colors $\sC$ is also $\Z_n$. We let  the color of an edge $\{i, j\}$ be $(i + j) \in \Z_n$  where  of course the addition $+$  means  addition modulo $n$ in $\Z_n$.  More precisely,  the edge-coloring in $G_n$ is given by $c: E \to  \sC$ is defined by $c (\{i, j\}) = i + j$ for all $i, j \in \Z_n$.

\begin{prop}
{\rm The edge-colored graph $G_n$ is a uniform graph.}
\end{prop}

\proof Let $i, j, k \in \Z_n$ be distinct. Then if $c (\{i, j\})  = c (\{i, k\})$, then $i + j = i + k$. Hence $j = k$.  This shows that no two  edges incident on the same vertex  have the same color.

Consider the group homomorphism $f : \Z_n \to \Z_n$ defined by $f(i) = i + i = 2i$. Since $n$ is odd, $f$ is injective. For, if $i \in \Z_n$ with $2i  = {0}$, then the order of $i$ is either 1 or 2 and divides the odd number $n$. Hence $i = {0}$ and $f$ is a group isomorphism.

For $m \in \Z_n$, we denote the set of all edges with color $m$ by $A_m$. Equivalently,  \[A_m  = \{ \{i, j\}  \, : \,i, j \in \Z_n,  m = i + j,\,  i \neq j\}. \] 

  We will prove that $|A_m| = \frac{n-1}{2}$.  Since $f$ is a bijection, there is a unique $l \in \Z_n$ such that $2 l = m$. Hence 
   $A_m  = \{ \{i, m-i\}  \, : \, i \in \Z_n,\, i \neq l\}$. Note that for each $i \in \Z_n$, we have  $\{ i, m-i\} = \{m-i, i \} $. Therefore,  $|A_m| = \frac{n-1}{2}$. In other words, the number of edges with color $m$ is constant for all $m$. This proves that the edge-colored graph $G_n$ is uniform. \hfill $\square$

\subsection{Color Permuting Automorphism Group of $G_n$}
In this section, we study the structure of the color permuting automorphism group $\CPA(G_n)$ of the edge-colored graph $G_n$ and  prove Theorem \ref{main1}.

\begin{defn}{\label{special}}
{\rm We call a bijection $\tau: \Z_n \to \Z_n$ {\em special} if for all $a, b, c, d \in \Z_n$ with $a \neq b, c \neq d$, and $a+ b = c+ d$, we have  $\tau(a) + \tau(b) = \tau(c) + \tau(d)$.}
\end{defn}

We first observe the following.

\begin{prop}\label{equivalent}
{\rm The following statements are equivalent for a bijection $\tau: \Z_n \to \Z_n$.
\begin{enumerate}
\item $\tau$ is special. 

\item For all $a, b, c, d \in \Z_n$ with  $c \neq d$, and $a+ b = c+ d$, we have  $\tau(a) + \tau(b) = \tau(c) + \tau(d)$.

\item For all $a, b, c, d \in \Z_n$ with $a- c = d-b$, we have  $\tau(a) - \tau(c) = \tau(d) - \tau(b)$.

\item $\tau \in \CPA(G_n)$. 
\end{enumerate}

}
\end{prop}

\proof Assume (1).  Let $a, c, d \in \Z_n$ and assume that $a+a = c+d$. Let $l = \tau(c) + \tau(d)$ We will prove that $\tau(a) + \tau(a) = l$. 
 Let $B = \mathbb Z_n \setminus \{a\}$. If $x \in B$, then $x \neq 2a - x $ as $n$ is odd, and $x + (2a - x) = c + d$. By our assumption (1), we have $\tau(x) + \tau(2a - x) = \tau(c) + \tau(d) = l.$ From this, we can conclude that $\tau(B) = \{ l - \tau(2a - x) \, :\, x \neq a\} = \Z_n \setminus \{l - \tau(a)\}$.  Since $\tau$ is a bijection, this implies that $\tau(a) = l - \tau(a)$ and hence $\tau(a) + \tau(a) = l$ which proves (2).
 
Assume (2). Then the statement (3) is clear for elements $a, b, c, d \in\Z_n$ with $c \neq d$.  The case $a \neq b$ follows similarly. 
If  $a=b$ and $c = d$ and  $a- c = d- b$, then $a= b= c=d$ as $n$ is odd.  Hence  $\tau(a) - \tau(c) = \tau(c) - \tau(a) = 0$.  This proves (3).

It is clear that (3) $\implies$ (1). 

Assume (1). We define $\phi: \Z_n \to \Z_n$ as $\phi(m) = \sigma(i) + \sigma(j)$ where $m = i + j$ and $i \neq j$. We note that $\phi$ is well-defined function because $\tau$ is special. 
Let $l \in \Z_n$. We write $l = {a} + {b}$ where ${a} \neq {b}$. Then $\phi(\tau^{-1}({a}) + \tau^{-1}({b})) = {a} + {b} = l$. Hence $\phi$ is surjective and hence it is a bijection. Also the color of an edge $\{\tau(i) ,\tau(j)\}$ is the same as $\phi(i + j)$, i.e. $\phi$(color of the edge $\{i, j\})$. This proves that $\tau \in \CPA(G_n)$. Hence (1) $\implies$ (4). 

Suppose now $\tau \in \CPA(G_n)$. Then there exists a permutation $\phi$ of $\Z_n$ such that the color of the edge $\{\tau(a), \tau(b)\}$ is the same as $\phi$(color of the edge $\{a, b\}$). Equivalently, $\tau(a) + \tau(b) = \phi (a + b)$. In particular,  if $a+ b = c + d$, then $\tau(a) + \tau(b) = \tau(c) + \tau(d)$. Hence $\tau$ is special. Hence (4) $\implies$ (1).  \hfill$\square$

\vspace{0.5cm}

Next we define a notion of an affine bijection on an abelian group.

\begin{defn}\label{affine}
{\rm   Let $(A, +)$ denote an abelian group. A bijection $f : A \to A$ is called {\em affine} if there exists a group automorphism $\gamma_f$ of $A$ such that  \[f(a + x) = \gamma_f(a)+ f(x)\] for all $a, x \in A$. We denote the set of all affine bijections on $A$ by $\Aff(A)$.}
\end{defn}

It is not difficult to check that $\Aff(A)$ is a group under composition. 

\begin{prop}\label{specialaff}
{\rm Let $f: \Z_n \to \Z_n$ be a bijection. Then the following statements are equivalent.
\begin{enumerate}
\item $f \in \CPA(G_n)$.

\item $f$ is special. 

\item $f \in \Aff(\Z_n)$.

\end{enumerate} }
\end{prop}

\proof  (1) $\iff$ (2) by Proposition \ref{equivalent}.

Assume that $f$ is special. We define $\gamma_f: \Z_n \to \Z_n$ by $\gamma_f({a}) = f({a}) - f({0})$ for all ${a} \in \Z_n$.
First we prove that $\gamma_f$ is a group homomorphism.  Let $i, j \in \Z_n$. As $f$ is special, by Proposition \ref{equivalent}, for all $i, j \in \Z_n$,  we have $f(i+j) - f(j) = f(i) - f ({0}).$ 

Hence for all $i, j \in \Z_n$, 
\begin{align*}
\gamma_f(i+j) &= f(i + j) - f ({0})\\
 &= f(i + j) - f(j) + f(j) - f({0})\\
 &= f(i) - f({0}) + f (j) - f ({0})\\
 &=\gamma_f(i) + \gamma_f(j).
\end{align*}

Suppose that $\gamma_f(i) = {0}$. This implies that $f(i) - f({0}) = {0}$ $\implies$ $f(i) = f({0})$. As $f$ is one-to-one, $i = {0}$. Hence $\ker \gamma_f = \{{0}\}$. This proves that  $\gamma_f$ is a bijection and hence a group automorphism of $\Z_n$. 

Also for ${a}, {x} \in \Z_n$, we have $f({a} + {x}) - f({x}) = f({a}) - f({0})$ by Proposition \ref{equivalent}. Hence $f({a} + {x}) - f({x}) =\gamma_f({a})$ and $f({a} + {x})  =\gamma_f({a}) +  f({x})$. This proves that $f \in \Aff(\Z_n)$ and (2) $\implies$ (3). 

We will prove that (3) $\implies$ (2). For, we assume that $f \in \Aff(\Z_n)$ and let $\gamma_f \in \Aut(\Z_n)$ such that $f({a} + {x})  =\gamma_f({a}) +  f({x})$ for all ${a}, {x} \in \Z_n$. Let $i, j, k, l \in \Z_n$ with $i -j = k - l$. We will prove that $f(i) - f(j) = f(k) - f(l)$.  
\begin{align*}
f(i) - f(j) &= f((i - j) + j) - f(j) \\
 &= \gamma_f(i - j) \\
 &= \gamma_f(k - l)\\\
 &= f((k - l) + l) - f(l)\\
 &= f(k) - f(l).
\end{align*}
By Proposition \ref{equivalent}, $f$ is special. \hfill$\square$

\begin{prop}\label{propaff}
{\rm   $\Aff(\Z_n) \cong \Z_n \rtimes \Aut(\Z_n)$.}
\end{prop}
\proof  We define $\phi: \Aff(\Z_n) \to \Aut(\Z_n)$ by $\phi(f) = \gamma_f$ where $\gamma_f \in \Aut(\Z_n)$ such that $f({a} + {x}) = \gamma_f({a}) + f({x})$ for all ${a}, {x} \in \Z_n$. We note that if $f \in \Aff(\Z_n)$, then for all ${x} \in \Z_n$, $\gamma_f({x}) = f({x}) - f({0})$. We first prove that $\phi$ is a group homomorphism. Let $f, g \in \Aff(\Z_n)$. 
We need to prove $\gamma_{f\circ g} = \gamma_f \circ \gamma_g$. Let ${a} \in \Z_n$. 
Then \begin{align*}
\gamma_f(\gamma_g({a})) &= \gamma_f(g({a}) - g({0}))\\
&=\gamma_f(g({a})) - \gamma_f(g({0})) \,\,\text{as } \gamma_f \in \Aut(\Z_n)\\
&= f(g({a})) -f({0}) - (f(g({0})) - f({0}))\\
&= f\circ g({a}) - f \circ g({0})\\
&= \gamma_{f \circ g}({a}).
\end{align*}
This shows that $\phi(f \circ g) = \phi(f) \circ \phi(g)$ and hence $\phi$ is a group homomorphism.

It is clear that $\ker \phi = \{f \in \Aff(\Z_n) \, : \, f({x}) = {x} + f({0}) \text{ for all } {x} \in \Z_n \}$.  Equivalently,  $\ker \phi = \{T_{{a}} \,:\, {a} \in \Z_n\}$ where $T_{{a}}: \Z_n \to \Z_n$ is a translation by ${a}$ given by $T_{{a}}({x}) = {x} + {a}$.  For, if $f \in \ker \phi$, then $f = T_{f({0})}$. Also given ${a} \in \Z_n$, we have $T_{{a}} \in \Aff(\Z_n)$ as $T_{{a}}({x} + {y}) = {x} + {y} + {a} = {x} + T_{{a}}({y})$ and hence $T_{{a}} \in \ker \phi$.  We note that $\ker \phi \cong \Z_n$. 

If $\gamma \in \Aut(\Z_n)$, then $\gamma({a} + {x}) = \gamma({a}) + \gamma({x})$ for all ${a}, {x} \in \Z_n$ $\implies$ $\gamma \in \Aff(\Z_n)$ and $\phi(\gamma) = \gamma$. Hence $\phi$ is surjective and $\phi \circ \iota = \id$ where $\iota: \Aut(\Z_n) \to \Aff(\Z_n)$ is the inclusion and $\id: \Aut(\Z_n) \to \Aut(\Z_n)$ is the identity map. This means that the following exact sequence splits.
\[0 \longrightarrow \Z_n \longrightarrow \Aff(\Z_n) \xrightarrow[]{\text{  } \phi \text{  }} \Aut(\Z_n) \longrightarrow 1\]

This proves that $\Aff(\Z_n) \cong \Z_n \rtimes \Aut(\Z_n)$. \hfill$\square$


\begin{proof}[Proof of Theorem \ref{main1}] By Proposition \ref{specialaff}, $\CPA(G_n) = \Aff(\Z_n)$ and by Proposition \ref{propaff},  $\Aff(\Z_n)$ is isomorphic to $\Z_n \rtimes \Aut(\Z_n)$. 
\end{proof}

\begin{remark}\label{Stab}
{\rm We note that $\CPA(G_n)$ acts transitively on the set of vertices $\Z_n$ as all rotations are the color permuting automorphisms of $G_n$.
 The stabilizer of ${0}$ under this action is precisely the automorphisms group of $\Z_n$, $\Aut(\Z_n)$. For if $f \in \CPA(G_n)$ and $f({0}) = {0}$, then  there exists $\gamma_f \in \Aut(\Z_n)$ such that $f( {a} + {x}) = \gamma_f({a}) + f({x})$ for all ${a}, {x} \in \Z_n$. Then for ${a} \in \Z_n$, we have $f({a}) = f({a} + {0}) = \gamma_f({a}) + f({0}) = \gamma_f({a})$. This shows that $f = \gamma_f \in \Aut(\Z_n)$.}
 \end{remark}

\section{2-step nilpotent Lie Algebras}

\subsection{Associating a Lie algebra with a graph}\label{sec-graph}

In this section we recall the construction of a 2-step nilpotent Lie algebra associated with an edge-colored {\em directed } graph (see \cite{R} and also \cite{PS}). 
Consider an edge-colored  directed simple graph  $H=(S,E,c: E \rightarrow \mathscr{C})$ where $S$ is the set of vertices, $E$ is the set of directed edges, and $c$ is a surjective edge-coloring function from the set of (directed) edges to the set of colors $\mathscr{C}$. We will denote a directed edge from $\alpha$ to $\beta$ by an ordered pair $(\alpha, \beta)$. By abuse of notation, we will  denote the color of the directed edge $(\alpha, \beta) \in E$, by simply $c(\alpha, \beta)$  rather than the more accurate $c\left((\alpha, \beta)\right)$. 

We  associate with $H$ a 2-step nilpotent Lie algebra $\N_H$ over $\mathbb R$ in the following way.
  The underlying vector space of $\N_H$ is $V \oplus W$ where $V$ is the $\mathbb R$-vector space consisting of formal $\mathbb R$-linear combinations of elements of $S$ (so that $S$ is a basis of $V$), and $W$ is the $\mathbb R$-vector space consisting of formal $\mathbb R$-linear combinations of elements of $\mathscr{C}$. The Lie bracket structure on $\mathcal{N}_H$ is given by the following 
\begin{enumerate}
\item If  $(\alpha, \beta) \in E$ and  $c(\alpha,  \beta) = Z $, then $[\alpha,\beta]=-[\beta, \alpha]=Z.$
\item If  $(\alpha, \beta) \notin E$, then $[\alpha, \beta] = [\beta, \alpha] = 0.$
\item $[Y,Z]=0$ for all $Y \in \N_H$ and $Z \in W$.
 
\end{enumerate}

We say that $\mathcal{N}_H$ is the {\em 2-step nilpotent Lie algebra} associated with the graph $H$. Note that the derived Lie algebra $[\N_H, \N_H]$ is the span of $\C$ and the dimension of $\N_H$ is $|S| + |\C|$. 

The above  construction is a generalization of the construction of 2-step nilpotent Lie algebras associated with simple graphs as in \cite{DM, M} where the  edge-coloring  $c$ is a bijection.

\begin{ex}\label{examplenilpotent}
{\rm Consider the following directed edge-colored graph $H=(S,E,c: E \rightarrow \mathscr{C})$, where $S= \{\alpha, \beta, \gamma, \delta\}$, $E= \{(\alpha, \beta), (\beta, \gamma), (\gamma, \delta), (\delta, \alpha)\}$, $\C = \{Z_1, Z_2\}$ and edge-coloring $c : E \to \C$ is given by
\[ c(\alpha, \beta) = c(\gamma, \delta) = Z_1, \]
\[c(\beta, \gamma) = c(\alpha, \delta) = Z_2.\]

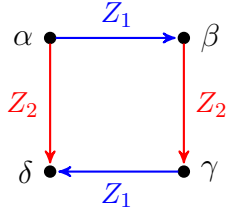
\begin{figure}[h!]
\begin{tikzpicture}[->,>=stealth',shorten >=1pt,auto,
  thick,vertex/.style={circle,draw,fill,scale=.35,font=\sffamily\large\bfseries},node distance=2in,thick]

\node[vertex, label=left:{$\alpha$}](X1) {};
  \node[vertex, label=right:{$\beta$}](X2) [right of=X1] {};
  \node[vertex, label=right:{$\gamma$}] (X3) [below of=X2]  {};
  \node[vertex, label=left:{$\delta$}] (X5)[below  of=X1]  {};

  \path[every node/.style={font=\sffamily\small}]
    (X1) edge ["$Z_1$",blue] (X2)
		(X2)     edge ["$Z_2$", red]  (X3)
			(X3)		edge ["$Z_1$", blue]  (X5)
	(X1) edge[red] node[left, red] {$Z_2$}   (X5)
				;
\end{tikzpicture}
\end{figure}
                              
Then the associated 2-step nilpotent Lie algebra $\N_H$ is of dimension 8. The only non-zero Lie brackets among the basis vectors of $\N_H$ are given by

\[[\alpha, \beta] = [\gamma, \delta] = Z_1 = -[\beta, \alpha] = -[\delta, \gamma], \]
\[[\beta, \gamma] = [\alpha, \delta] = Z_2 = -[\gamma, \beta] = - [\delta, \alpha]. \]

}
\end{ex}

\subsection{Automorphism group}
Recall that a  linear isomorphism $\tau: \mathcal{N} \to \mathcal{N}$ is called a {\em Lie automorphism} of the Lie algebra $\mathcal{N}$ if for all $X, Y \in \mathcal{N},$ \[\tau[X, Y] = [\tau(X), \tau(Y)].\]   The group of all automorphisms of the Lie algebra $\mathcal{N}$ 
 is called the {\em automorphism group of $\mathcal{N}$ } and is denoted by $\Aut(\mathcal{N})$.


Given an edge-colored directed graph $H = (S, E, c: E \to \mathscr{C})$, we will characterize those  graph automorphisms of the simple graph $(S, E)$ which can be extended to  Lie automorphisms of the associated 2-step nilpotent Lie algebra $\N_H$.

Let $H = (S, E, c: E \to \mathscr{C})$ be an edge-colored simple directed graph. Let $E_u$ be the collection of associated undirected edges;
\[ E_u =\{\{\alpha, \beta \}\,\,:\,\, (\alpha, \beta) \in E \}.\] 
We color the undirected edges by the same colors, i.e. we use the coloring $c_u : E_u \to \C$ given by \[c_u(\{\alpha, \beta\}) = c(\alpha, \beta) \text{ if } (\alpha, \beta) \in E.  \]

The edge-colored undirected simple graph $H_u = (S, E_u, c_u : E_u \to \C)$ will called the {\em underlying undirected graph} of $H$. 

Let $E^{-}$ denote the set $\{(\beta, \alpha) :  (\alpha, \beta) \in E\}$ and  $\mathscr{C}^{-}$ denote the set $\{-Z \in W : Z \in \mathscr{C}\}$. We extend the edge-coloring function $c$ on $E \cup E^{-}$ as follows:  If $c(\alpha, \beta) = Z$, then we define  $c(\beta, \alpha)$ by $-Z$.

\begin{defn}\label{GLA}
  {\rm Let $H = (S, E, c: E \to \mathscr{C})$ be an edge-colored simple directed graph. We say that a color permuting automorphism $\chi$ of the underlying undirected graph $H_u= (S, E_u, c_u: E_u \to \mathscr{C})$ is a {\em graph Lie automorphism} of $H$ if  it induces a permutation on $\mathscr{C} \cup \mathscr{C}^{-}$, i.e.,  if there exists a permutation $\phi$ of $\mathscr{C} \cup \mathscr{C}^{-}$ such that $c(\chi(\alpha), \chi(\beta)) = \phi (c(\alpha, \beta))$ for all $(\alpha, \beta) \in E \cup E^{-}$.  We denote the group of all such automorphisms by $\GLA(H)$.}
\end{defn}

\begin{ex}\label{exampleGLA}
{\rm Consider the following directed edge-colored graph $H$  and the associated 2-step nilpotent Lie algebra $\N_H$ as in Example \ref{examplenilpotent}. 
\begin{figure}[h!]
\begin{tikzpicture}[->,>=stealth',shorten >=1pt,auto,
  thick,vertex/.style={circle,draw,fill,scale=.35,font=\sffamily\large\bfseries},node distance=2in,thick]

\node[vertex, label=left:{$\alpha$}](X1) {};
  \node[vertex, label=right:{$\beta$}](X2) [right of=X1] {};
  \node[vertex, label=right:{$\gamma$}] (X3) [below of=X2]  {};
  \node[vertex, label=left:{$\delta$}] (X5)[below  of=X1]  {};

  \path[every node/.style={font=\sffamily\small}]
    (X1) edge ["$Z_1$",blue] (X2)
		(X2)     edge ["$Z_2$", red]  (X3)
			(X3)		edge ["$Z_1$", blue]  (X5)
	(X1) edge[red] node[left, red] {$Z_2$}   (X5)
				;
\end{tikzpicture}
\end{figure}
The color permuting automorphism group of the  underlying undirected graph $H_u$, $\CPA(H_u) \cong D_8$ (see Example \ref{4uniform}). 
If $\chi = (\alpha \,\beta) (\gamma\, \delta)$, then we can see that $\chi \in \CPA(H_u)$. Note that $\C \cup \C^{-} = \{Z_1, Z_2, -Z_1, -Z_2\}$. We define a permutation $\phi$ of $\C \cup \C^{-}$ as follows:
\[\phi(Z_1) = -Z_1,\, \phi(Z_2) = Z_2,\, \phi(-Z_1) = Z_1, \,\phi(-Z_2) = -Z_2.\]
Then $c(\chi(\alpha), \chi(\beta)) = c(\beta, \alpha) = -Z_1 = \phi(Z_1) = \phi(c(\alpha, \beta))$.  Similarly one can check that $c(\chi(x), \chi(y)) = \phi(c(x, y))$ for all $(x, y) \in E \cup E^{-}$. Hence $\chi \in \GLA(H)$. 

Now if $\sigma  = (\alpha \,\beta\, \gamma\, \delta)$, then $\sigma \in \CPA(H_u)$. Note that $\sigma \notin \GLA(H)$. This is because $c(\sigma(\alpha), \sigma(\beta)) = c(\beta, \gamma) = Z_2$ and $c(\sigma(\gamma), \sigma(\delta)) = c(\delta, \alpha) = -Z_2$. However $c(\alpha, \beta) = c(\gamma, \delta) = Z_1$ and hence there is no permutation $\psi$ of $\C \cup \C^{-}$ such that $c(\sigma(\alpha), \sigma(\beta)) = \psi(c(\alpha, \beta))$. 
}
\end{ex}

We now show that the elements of $\GLA(H)$ give rise to automorphisms of the associated Lie algebra $\N_H$.

\begin{lemma}\label{GLAextension}

{\rm Let $H = (S, E, c: E \to \mathscr{C})$ be an edge-colored simple directed graph.  If  $\chi \in \GLA(H)$, then $\chi$ can be uniquely extended to a Lie  automorphism of $\N_H$. Therefore, the group $\GLA(H)$ can be realized as a subgroup of $\Aut(\N_H)$.}
\end{lemma}

\begin{proof}

Note that $\N_H = V \oplus W$ where $V$ is the $\mathbb R$-vector space with $S$ as a basis and $W$ is  the $\mathbb R$-vector space with $\C$ as a basis.  In order to extend $\chi$ to a Lie automorphism of $\N_H$, we first extend $\chi$ linearly on $V$ and then  linearly on $W$ by defining $\chi(Z) = c(\chi(\alpha), \chi(\beta))$ if $Z = c(\alpha, \beta)$. We will denote the extended linear map from $\N_H$ to $\N_H$ by $\chi$ as well. 
Now  $\chi$ is well defined on $\N_H$ because $\chi \in \GLA(H)$ (see Definition \ref{GLA}). It can be seen that $\chi$ is onto and hence it is a linear isomorphism. 

If $(\alpha, \beta) \in E$ and $c(\alpha, \beta) = Z$, then $\chi([\alpha, \beta]) = \chi(Z) =  c(\chi(\alpha), \chi(\beta)) =  [\chi(\alpha),  \chi(\beta)]$ by definition of the Lie bracket on $\N_H$.  As $\chi$ is linear, we have $\chi([X_1, X_2]) = [\chi(X_1), \chi(X_2) ]$ for all $X_1, X_2 \in V$. Recall $[Y, U] = 0$  for all $Y \in \N_H$ and $U \in W$.  Using the linearity of $\chi$ again, we have   \[\chi([X, Y]) = [\chi(X), \chi(Y) ]\] for all $X, Y \in \N_H$ and $\chi \in \Aut(\N_H)$.  The uniqueness of the extension is clear from the definition of $\N_H$.
\end{proof}

\subsection{The directed graph $H_n$}\label{sectionedgedirection}   
Throughout we assume that $n$ is an odd integer. We  define the directed edge-colored graph $H_n$  whose underlying undirected graph is $G_n$ as introduced in Section \ref{G_n}. We define the vertex set of  $H_n$  to be $\Z_n =\{0, 1, \ldots, n-1\}$ and 
 the directed edge set $E$ as follows:  \[\displaystyle E = \left\{( m+i,\,\,  m-i) \,\,:\,\,  0 \leq m \leq n-1,\,\,1\leq i \leq \dfrac{n-1}{2}\right\}.\]  
The set of colors $\C$ is denoted by $\{Z_{i}\,\,:\,\,i \in \Z_n\}$ and the edge-coloring  $c: E \to  \C$ is defined by $c(i, j) = Z_{i+j}$ for all $(i, j) \in E$.

Geometrically the orientation of edges of $H_n$ can be visualized as follows. Note that the underlying undirected graph $(H_n)_u$ is nothing but the graph $G_n$, which can be pictured as a regular $n$-gon in the plane along with all possible diagonals.
 To obtain $H_n$, we orient the $n$ edges of the regular polygon clockwise and then orient the diagonals in such a way that the edges and diagonals which are parallel receive the same orientation.

For example, the directed edge-colored graph $G_5$ is as below.

\begin{figure}[h!]
\begin{tikzpicture}[->,>=stealth',shorten >=1pt,auto,
  thick,vertex/.style={circle,draw,fill,scale=.35,font=\sffamily\large\bfseries},node distance=3in,thick]

\node[vertex, label=above:{${0}$}](X0)  at (90:3) {};
  \node[vertex, label=right:{${1}$}](X1) at (18:3) {};
 \node[vertex, label=right:{${2}$}] (X2)  at (306:3)  {};
 \node[vertex, label=left:{${3}$}] (X3) at (234:3) {};
	 \node[vertex, label=left:{${4}$}] (X4) at (162:3) {};

\path[every node/.style={font=\sffamily\small}]
    (X0) edge ["$Z_{{1}}$",blue] (X1)
					 (X1) edge ["$Z_{{3}}$",red] (X2)
		 (X2) edge ["$Z_{{0}}$",green] (X3)
         (X2) edge ["$Z_{{0}}$",green] (X3)
         (X3) edge ["$Z_{{2}}$"] (X4)
         (X4) edge ["$Z_{{4}}$", orange] (X0)
         (X4) edge ["$Z_{{4}}$", orange] (X0)
        (X4) edge [blue]  node[sloped, below] {$Z_{{1}}$}(X2)
  (X0) edge [red] node[sloped, above] {$Z_{{3}}$} (X3)
	(X1) edge [green] node[sloped, above] {$Z_{{0}}$}  (X4)	
    (X1) edge [green] node[sloped, above] {$Z_{{0}}$} (X4)	
    (X2) edge []  node[sloped, above] {$Z_{{2}}$}(X0)
     (X3) edge [orange]  node[sloped, below] {$Z_{{4}}$}(X1)
     (X3) edge [orange]  node[sloped, below] {$Z_{{4}}$}(X1)
    ;
\end{tikzpicture}
\caption{$H_5$}\label{fig-H5}
\end{figure}

\subsection{Graph Lie automorphism group of $H_n$} 

First, note that $\GLA(H_n)$ acts transitively on the set of vertices $\Z_n$. This is because rotations (or translations) are the graph Lie automorphism. To see this, suppose that  $\chi :\Z_n \to \Z_n$ is given by $\chi(i) = i + k$ for all $i \in \Z_n$. We define a permutation $\phi$   on $\mathscr{C}$ by $\phi(Z_{l}) = Z_{ l+2k}$ for all $l \in \Z_n$. Then if $0 \leq m \leq n-1$ and $1\leq i \leq \frac{n-1}{2}$, we have 

\begin{align*}
c \circ \chi (( m+i,\,\,  m-i) &=c( m+i+2k,\,\,  m-i+2k)\\
&= Z_{ 2m+4k}\\
&= \phi(Z_{ 2m+2k})\\
&= \phi \circ c (( m+i,\,\,  m-i).
\end{align*}
 
We can extend $\phi$ on $\mathscr{C} \cup \mathscr{C}^{-}$ by defining $\phi(-Z_{l}) = -Z_{ l+2k}$.  Hence $\chi \in \GLA(H_n)$. 

Let $\Stab({0})$ denote the stabilizer of ${0}$ under the action of $\GLA(H_n)$ on $\Z_n$, i.e. let \\ $\Stab({0}) = \{ \sigma \in \GLA(H_n) \,\,:\,\, \sigma({0}) = {0}\}.$
Then by Orbit-stabilizer theorem,  we have 
\begin{equation} \label{OSE}
|\GLA(H_n)| = n |\Stab({0})|.
\end{equation}

We will prove that $\Stab({0}) = \{\pm \id\}$. In other words, we will prove that if $\id \neq \chi \in \GLA(H_n)$ and $\chi({0}) = {0}$, then $\chi$ is a reflection, i.e. $\chi(i) =  n-i$ for all $i \in \Z_n$.

 Let  $R = \{{1}, \ldots,  \frac{n-1}{2} \} \subset \Z_n$ and  $L = \{  \frac{n+1}{2}, \ldots,  n-1\} \subset \Z_n$.   Note that in Figure \ref{fig-H5} for $H_5$, $R$ (resp. $L$) consists of the vertices to the right (resp. left) of the vertical line through $0$.  

\begin{lemma}\label{lemma1}
{\rm If $\tau \in \Aut(\Z_n)$ and $\tau \neq \id$, then $\tau(R) \neq R$.  }
\end{lemma}
\proof Let $\tau({1}) = k$. If $k \in L$, we are done as ${1} \in R$. Now we assume that $k \in R$, i.e. we assume that $2 \leq k \leq \frac{n-1}{2}$. We claim that there exists $q$ with $1 \leq q \leq \frac{n-1}{2}$  such that $\frac{n+1}{2} \leq qk \leq n-1$. 
In $\Z$, we divide $n-1$ by $k$. Let $q \in \mathbb N$ and $r$ with $0 \leq r \leq k-1$ such that $n-1 = q k + r$. 
Then $q = \frac{n-1}k - \frac{r}k \leq \frac{n-1}k \leq \frac{n-1}{2}$ as $k \geq 2$. Hence $q \leq \frac{n-1}{2}$. 
Also $ n-1 \geq n-1-r  > \frac{n-1}{2}$ as $r < \frac{n-1}{2}$.  Hence $ n-1 \geq q k   > \frac{n-1}{2}$. This proves our claim.

As $\tau \in \Aut(\Z_n)$ and $\tau({1}) = k$, we have $\tau({q}) =  qk$. Hence $\tau(R) \neq R$ as ${q} \in R$ and $ qk \in L$. \hfill$\square$

\begin{lemma}\label{lemma2}
{\rm   Suppose that $ \chi \in \GLA(H_n)$  with  $\chi({0}) = {0}$. If $\chi(k) \in R$ for some $k \in R$, then $\chi(R) = R$. }
\end{lemma}

\proof Assume that $\chi(k) = i \in R$ where $k \in R$. We note that $\chi$ is  a color permuting automorphism of the undirected edge-colored complete graph.   By Remark \ref{Stab},  $\chi \in \Aut(\Z_n)$. Hence $\chi( n-k) =  n-i \in L$. 
Assume that $\chi(j) \in L$ for some $j \in R$. 
Then $\chi( n-j) \in R$.  Hence $c \circ \chi((k,  n-k) = c(i,  n-i) = Z_{0}$ and $c \circ \chi(j,  n-j) = c(\chi(j), \chi( n-j) = -Z_{{0}}$.  We note that $c(k,  n-k) = c(j,  n-j)$ as both $k$ and $j$ are in $R$.   This is a contradiction to our assumption that $\chi \in \GLA(H_n)$. \hfill$\square$

\begin{prop}\label{last}
{\rm     If $\chi \in \GLA(H_n)$ and $\chi({0}) = {0}$, then $\chi^2 = \id$.       }
\end{prop}

\proof If $\chi \in \GLA(H_n)$ and $\chi({0}) = {0}$, then $\chi \in \Aut(\Z_n)$ by Remark \ref{Stab}. By Lemma \ref{lemma1} and Lemma \ref{lemma2},  we have $\chi(L) = R$ and $\chi(R) = L$.  We note that $\chi^2 \in \GLA(H_n)$ and hence $\chi^2(L) = \chi(R) = L$, $\chi^2(R) = \chi(L) = R$. 
By Lemma \ref{lemma1}, $\chi^2 = \id$. \hfill$\square$

As noted before, $\GLA(H_n)$ contains all $n$ rotations. Also  Proposition \ref{last} implies that the only non identity group automorphism which is a graph Lie automorphism must be the reflection about ${0}$. This proves Theorem \ref{main2}.

\end{document}